\documentclass[12pt]{article}
\usepackage{amsmath, amsfonts, amssymb, amsthm}

\usepackage[ps2pdf,linktocpage,bookmarks,bookmarksnumbered,bookmarksopen]{hyperref}
\hypersetup{%
   pdftitle   ={Path regularity of Gaussian processes via small deviations},
   pdfsubject ={Manuscript},
   pdfauthor  ={Frank Aurzada},
   pdfkeywords={Small deviation; small ball probability; Gaussian process; sample path regularity}
      }

\let\BFseries\bfseries\def\bfseries{\BFseries\mathversion{bold}} % formulas in headings bold

\newcommand{\eps}{\varepsilon}
\theoremstyle{plain}
\newtheorem{thm}{Theorem}

\newtheorem{prop}[thm]{Proposition}
\newtheorem{cor}[thm]{Corollary}
\theoremstyle{definition}

\newtheorem{rem}[thm]{Remark}
\newtheorem{exa}[thm]{Example}
\renewenvironment{proof}[1][] {\smallskip \noindent {\bf Proof#1:} }{\hspace*{\fill}$\square$\medskip\par}
\def\P{{\bf {\mathbb{P}}}}
\newcommand{\pr}[1]{\P\left[#1\right]}

\def\E{\mathbb{E}} 
\newcommand{\norm}[1]{\left\|#1\right\|}

\def\R{\mathbb{R}}

\newcommand{\gap}{\gtrsim}\newcommand{\lap}{\lesssim}
\newcommand{\lew}{\precapprox}\newcommand{\gew}{\succapprox}

\newcommand{\dM}{M}
\newcommand{\dif}{\mbox{d}}
\def\dealpha{\tau}
\def\debeta{\theta}

\begin{document}
\title{Path regularity of Gaussian processes via small deviations}
\author{Frank Aurzada}
\date{May 20, 2009}
\maketitle
\begin{abstract} We study the a.s.\ sample path regularity of Gaussian processes. To this end we relate the path regularity directly to the theory of small deviations. In particular, we show that if the process is $n$-times differentiable then the exponential rate of decay of its small deviations is at most $\varepsilon^{-1/n}$. We also show a similar result if $n$ is not an integer.
\end{abstract}
\noindent{\slshape\bfseries Running Head:} Path regularity via small deviations

\section{Introduction}
The small deviation problem, also called small ball problem, consists in determining the rate of increase of the quantity
\begin{equation} \label{sd}
-\log \pr{ \norm{X} \leq \eps }, \qquad \text{as $\eps\to 0$.} 
\end{equation}
Here, $X$ is a random variable with values in a normed space $(E,\norm{.})$.

This problem has several connections to approximation quantities for stochastic processes. We refer to \cite{lishao-overview} for an overview of the field and links to applications and to \cite{sdreferences} for a regularly updated list of references. Recently, several articles have focused on the small deviation problem for integrated Gaussian processes, \cite{khoshshi,chenli,gaohannigtorcasso,gaoetal,filletal,nazarovnikitin,lifshitssimon2004,gaomr}. This is mainly due to the connection of the problem to the spectral asymptotics of certain boundary value problems.

In some sense, this paper also considers integrated processes. However, we do not aim at finding the rate in (\ref{sd}) for Gaussian processes, but rather at showing that this rate is directly related to the sample path regularity of the process.

The idea that the small deviation rate encodes the smoothness properties of the Gaussian process has been present in many articles on small deviations. However, it seems that no concrete results are available that relate the small deviation rate \emph{directly} to smoothness properties of the process, e.g.\ differentiability. The aim of this article is to provide this direct link.

In particular, we are going to show (Corollary~\ref{cor:nregularity} below) that if a Gaussian process is $n$-times differentiable then for its small deviation rate
$$
-\log \pr{ \norm{X}_{L_\infty[0,1]} \leq \eps } \leq c \,\eps^{-1/n}, \qquad \text{as $\eps\to 0$.} 
$$
We also show a similar result if $n$ is not an integer. As we shall see, this provides sharp criteria for the path regularity of Gaussian processes. Note that for example for Brownian motion the small deviation rate is $\eps^{-2}$, which corresponds to being H\"{o}lder continuous up to $\frac{1}{2}$. A consequence is that for a $\mathcal{C}^\infty$ process we have that, for any $\delta>0$,
$$
\lim_{\eps\to 0} \eps^\delta\left( -\log \pr{ \norm{X}_{L_\infty[0,1]} \leq \eps }\right) = 0. 
$$

These results, combined with Li's weak decorrelation inequality \cite{lidecorr}, have one further consequence (Corollary~\ref{cor:remainderterm} below): If the Gaussian process $X$ can be represented as $X=Y+Z$, with $Y$ and $Z$ not necessarily independent, and $Z$ is smooth enough then $X$ and $Y$ have the same small deviation order. This can be used to show that the small deviation order of smoother `remainder terms' does not matter, as will be demonstrated with some examples.

For showing the relation between path regularity and small deviations we employ a result developed by Chen and Li in \cite{chenli}. They show the following (Theorem~1.2 in \cite{chenli}).

\begin{prop} \label{prop:chenli} Let $X$ be a centered Gaussian random variable with values in some separable Banach space $(E,\norm{.})$. Let $\mathcal{H}$ be the reproducing kernel Hilbert space of $X$ and denote by $|.|_\mathcal{H}$ the norm induced by the inner product in $\mathcal{H}$. Let $Y$ be another Gaussian random variable in $(E,\norm{.})$, not necessarily independent of $X$. Then, for any $\lambda, \eps>0$,
\begin{equation} \label{eqn:chenli1}
\pr{ \norm{Y} \leq  \eps} \geq \pr{ \norm{X} \leq  \lambda \eps} \E \exp\left( -\frac{\lambda^2}{2}\, |Y|_\mathcal{H}^2 \right).
\end{equation}

In particular, let $E$ be a space of functions from $[0,1]$ to $\R$ and let $B$ be Brownian motion. Then
\begin{equation} \label{eqn:clbm}
\pr{ \norm{Y} \leq  \eps} \geq \pr{ \norm{B} \leq  \lambda \eps} \E \exp\left( -\frac{\lambda^2}{2} \norm{Y'}_{L_2[0,1]}^2 \right),
\end{equation}
for any $\lambda, \eps>0$ and any Gaussian random variable $Y$ with values in $E$, where $Y'$ is the derivative of $Y$.
\end{prop}

This result was used to derive the small deviation rate for the $m$-fold integrated Brownian motion in \cite{chenli}. The procedure is as follows: Let $Y$ be integrated Brownian motion. Knowning the small deviation rate of Brownian motion $Y'$ gives the rate of the Laplace transform on the right-hand side in (\ref{eqn:clbm}), when $\lambda\to\infty$. The small deviation probability of Brownian motion $B$ w.r.t.\ $\norm{.}$ on the right hand side in (\ref{eqn:clbm}) is known as well. This gives a lower bound for the small deviation probability of  integrated Brownian motion $Y$ for basically any norm $\norm{.}$. The procedure can be iterated. On the other hand, the upper bound for the small deviation rate of $Y$ e.g.\ for the $L_\infty$-norm can be obtained simply by comparison to the easier $L_2$-norm.

The focus of the present note is
\begin{itemize}
 \item to formulate the idea from \cite{chenli} in a general framework (Section~\ref{sec:general}), and to extend it to fractional derivatives (Section~\ref{sec:frac}),
 \item to show that this leads to information on the a.s.\ path regularity of the Gaussian process under consideration (Section~\ref{sec:pr}), 
 \item to study a conditional version of (\ref{eqn:chenli1}) that can be applied in particular to stable processes (Section~\ref{sec:stable}), and
 \item to investigate relations to other questions (Sections~\ref{sec:eigenvalues}, \ref{sec:quantization}, \ref{sec:entropy}) and concrete examples (Section~\ref{sec:examples}).
\end{itemize}

In this paper, we let $X$ be a real-valued, centered Gaussian process indexed by $[0,1]$ with $X(0)=0$ a.s. The restriction $X(0)=0$ is for simplicity only. We use $\sim$, $\gap$, and $\lap$ for strong asymptotics, i.e.\ $f\gap g$ or $g\lap f$ if $\limsup f/g \leq 1$, $f \sim g$ if $\lim f/g =1$, whereas $\approx$, $\lew$, and $\gew$ stand for weak asymptotics, i.e.\ $f\gew g$ or $g\lew f$ if $\limsup f/g < \infty$ and $f\approx g$ if $f \lew g $ and $f\gew g$. We frequently use $1/0=\infty$ and $1/\infty=0$.

\section{Results for small deviations}
\subsection{First results} \label{sec:general}
Our first theorem concretizes the method used in \cite{chenli}. Here, $X'$ denotes the derivative of $X$.
\begin{thm} \label{thm:bmint} Let $0<\dealpha\leq \infty$, $\debeta\in\R$, and $1\leq p \leq \infty$. Then
\begin{equation} \label{eqn:ass1}
-\log \pr{ \norm{X'}_{L_2[0,1]} \leq  \eps} \lap \kappa \eps^{-1/\dealpha} |\log \eps|^{\debeta}
\end{equation}
implies
$$
-\log \pr{ \norm{X}_{L_p[0,1]} \leq  \eps} \lap C \eps^{-1/(\dealpha+1)} |\log \eps|^{\debeta\dealpha/(\dealpha+1)},
$$
where $C=C(\kappa,\kappa_p)$ and $\kappa_p$ is the small deviation constant of Brownian motion w.r.t.\ the $L_p$-norm.
\end{thm}

The proof goes along the lines outlined above; we skip it since it is included in the more general Theorem~\ref{thm:mostgen} below.

\begin{rem} \label{rem:constant}
We remark that the constant $C=C(\kappa,\kappa_p)$ can be computed explicitly. The constant is, in general, not the correct small deviation constant for $X$ w.r.t.\ $L_p$-norm. For $\dealpha=\infty$ (explicitly permitted in the above theorem) we get $C=\kappa$. The case $\dealpha=0$ is of special interest; and we treat it in Corollary~\ref{cor:nregularity}.
\end{rem}

Using the same idea as in Theorem~\ref{thm:bmint} we can obtain a converse estimate.

\begin{thm} \label{thm:converseone} Let $0\leq\gamma<1$, $\delta\in\R$ and $1\leq p \leq \infty$.  Then
$$
-\log \pr{ \norm{X}_{L_p[0,1]} \leq  \eps} \gap C \eps^{-\gamma} |\log \eps|^{\delta},
$$
implies
$$
-\log \pr{ \norm{X'}_{L_2[0,1]} \leq  \eps} \gap \kappa \eps^{-\gamma/(1-\gamma)} |\log \eps|^{\delta/(1-\gamma)}
$$
where $\kappa=\kappa(C,\kappa_p)$ and $\kappa_p$ is the small deviation constant of Brownian motion w.r.t.\ the $L_p$-norm.
\end{thm}

The cases $\gamma=0$ (included above) and $\gamma=1$ (treated in Corollary~\ref{cor:nregularity}) are of special interest.

\subsection{Fractional derivatives}  \label{sec:frac}
After demonstrating the method from \cite{chenli}, we now extend the idea to a more subtle situation. Here we define the fractional derivative as follows (cf.\ \cite{skm}). Recall that we work with processes with $X(0)=0$. For a given function $F$ with $F(0)=0$ and $\dM>0$, we set
\begin{equation} \label{eqn:fracder1}
\frac{\dif^{\dM} F}{\dif t^{\dM}}(t) = f(t)\qquad :\Leftrightarrow \qquad F(t)=\int_0^t (t-s)^{\dM-1} f(s) \, d s.
\end{equation}
We stress that the $M$-th derivative of $F$ ($M$ integer) coincides with $\dif^{\dM} F/ \dif t^{\dM}$.  If there is no ambiguity we also write for simplicity
$$
F^{(\dM)} = \frac{\dif^{\dM} F}{\dif t^{\dM}}.
$$

Now we get an analog to Theorem~\ref{thm:bmint}.

\begin{thm} \label{thm:hsb}
Let $0<\dealpha\leq \infty$, $\debeta\in\R$, $1\leq p \leq \infty$, and $M>0$. Then
$$
-\log \pr{ \norm{X^{(\dM)}}_{L_2[0,1]} \leq  \eps} \lap \kappa \eps^{-1/\dealpha} |\log \eps|^{\debeta}
$$
implies
$$
-\log \pr{ \norm{X}_{L_p[0,1]} \leq  \eps} \lap C \eps^{-1/( \dealpha+\dM)} |\log \eps|^{\debeta\dealpha/( \dealpha+\dM )},
$$
where $C=C(\kappa,\kappa_p^\dM)$ and $\kappa_p^\dM$ is the small deviation constant of a standard Riemann-Liouville process $R^H$ (cf.\ (\ref{eqn:defrl}) below) with $H=\dM-1/2$ w.r.t.\ the $L_p$-norm.
\end{thm}

The proof goes along the same lines as the one of Theorem~\ref{thm:bmint}. The only difference is that we use the Riemann-Liouville process $R^H$ instead of Brownian motion $B$. An analog to Theorem~\ref{thm:converseone} also holds true.

For $\dealpha=\infty$ (explicitly permitted in the above theorem) we get $C=\kappa$. Remark~\ref{rem:constant} applies accordingly.

\subsection{General result for translation invariant, self-similar, pseudo-additive norms}
In this section, we give the most general result for small devations. For this purpose, we recall from \cite{lifshitssimon2004} the notion of $\norm{.}$ being a translation invariant, $\beta$-self-similar, and $p$-pseudo additive functional semi-norm, for short $\norm{.}\in \mathbf{N}(\beta,p)$.
% Let $E$ be a linear spaces of functions from $I$ to $\R$.
%\begin{defn} Let $\beta\in \R$, $p\in[1,\infty]$. A norm $\norm{.}$ on $E$ is called $\beta$-self-similar, and $p$-pseudo additive functional norm, for short $\norm{.}\in \mathbf{N}(\beta,p)$, if
%\begin{itemize}
% \item $\norm{.}$ is contractive: $\norm{.}_J\leq \norm{.}_I$ if $J\subseteq I$,
% \item $\norm{.}$ is translation-invariant: $\norm{f}_{I-c}=\norm{f(.-c)}_I$ for $f\in E$, and $c\in \R$,
% \item $\norm{.}$ is $\beta$-self-similar: $\norm{f(c\cdot)}_{I/c} = c^\beta \norm{f}_I$ for any $f\in E$, and $c>0$,
% \item $\norm{.}$ is $p$-superadditive and $p$
%\end{itemize}
% \end{defn}
Here, we require that $\norm{.}$ be a true norm of a separable Banach space. Instead of rewriting the definition from \cite{lifshitssimon2004} we recall that the notion includes, for example, $L_p[0,1]$-norms ($\in\mathbf{N}(-1/p,p)$), $\eta$-H\"{o}lder norms ($\in\mathbf{N}(\eta,\infty)$), and the $p$-variation norm ($\in\mathbf{N}(0,p)$). Possibly it also includes certain Besov and Sobolev norms, see remarks in \cite{lifshitssimon2004}. Since $X(0)=0$, we are sure that we deal with a norm rather than a semi-norm.

\begin{thm} \label{thm:mostgen} Let $\norm{.}\in \mathbf{N}(\beta,p)$, $M>0$, $0<\dealpha\leq\infty$, and $\debeta\in\R$. Then
\begin{equation} \label{eqn:ass2}
-\log \pr{ \norm{X^{(\dM)}}_{L_2[0,1]} \leq  \eps} \lap \kappa \eps^{-1/\dealpha} |\log \eps|^{\debeta}
\end{equation}
implies
$$
-\log \pr{ \norm{X} \leq  \eps} \lap C \eps^{-1/( \dealpha+\dM -\beta-1/p)} |\log \eps|^{\debeta\dealpha/( \dealpha+\dM -\beta-1/p)},
$$
where $C=C(\kappa,\kappa_{\norm{.}}^\dM)$ and $\kappa_{\norm{.}}^\dM$ is the small deviation constant of a standard Riemann-Liouville process $R^H$ (cf.\ (\ref{eqn:defrl}) below) with $H=\dM-1/2$ w.r.t.\ $\norm{.}$.
\end{thm}

\begin{proof} Note that (\ref{eqn:ass2}) implies
$$
-\log \pr{ \norm{X^{(\dM)}}_{L_2[0,1]}^2 \leq  \eps} \lap \frac{\kappa}{2^\debeta}\, \eps^{-1/(2\dealpha)} |\log \eps|^{\debeta}.
$$
By de Bruijn's Tauberian theorem (Theorem~4.12.9 in \cite{bgt}), this implies
$$
-\log \E \exp\left( -\frac{\lambda^2}{2} \norm{X^{(\dM)}}_{L_2[0,1]}^2 \right) \lap K\, \lambda^{1/(\dealpha+1/2)} |\log \lambda|^{\debeta\dealpha/(\dealpha+1/2)},
$$
when $\lambda\to\infty$. Here $K$ can be computed explicitly from $\kappa$. We only remark that $\kappa=K$ for $\dealpha=\infty$. Note that the set
$$
\mathcal{H}:=\left\lbrace g : [0,1] \to \R ~\left|~ g(t)=\int_0^t (t-s)^{H-1/2} f(s)\, d s, f\in L_2[0,1], f(0)=0 \right.\right\rbrace
$$
with $H=M-1/2$ and norm
$$
\norm{g}_\mathcal{H} = \norm{f}_{L_2[0,1]}=\norm{g^{(M)}}_{L_2[0,1]}
$$
is the reproducing kernel Hilbert space (\cite{lifshits}) of the Riemann-Liouville process \begin{equation}\label{eqn:defrl}
R^H(t):=\int_0^t (t-s)^{H-1/2} \, d B(s),
\end{equation} where $B$ is a Brownian motion. Therefore, Proposition~\ref{prop:chenli} implies that
\begin{equation} \label{eqn:rlappl}
\pr{ \norm{X} \leq  \eps} \geq \pr{ \norm{R^H} \leq  \lambda \eps} \E \exp\left( -\frac{\lambda^2}{2} \norm{X^{(M)}}_{L_2[0,1]}^2 \right).
\end{equation}
We can use the results for the Riemann-Liouville process from \cite{lifshitssimon2004}, which yield that
\begin{equation} \label{eqn:rlresult}
-\log \pr{ \norm{R^H} \leq  \lambda \eps} \sim \kappa_{\norm{.}}^\dM (\lambda \eps)^{-1/(H-\beta-1/p)},
\end{equation}
as long as $\lambda \eps\to 0$. Set $\gamma:=1/(H-\beta-1/p)$. We use this with 
$$
\lambda:=D\eps^{-(\dealpha +1/2)/(1/\gamma+\tau+1/2)} |\log \eps|^{-\dealpha \debeta/(\gamma(1/\gamma+\tau+1/2))},
$$
where $D$ is some constant. This gives
\begin{multline*}
\lim_{\eps\to 0} \eps^{1/(1/\gamma+\dealpha+1/2)} |\log \eps|^{\dealpha \debeta/(1/\gamma+\dealpha+1/2)}\left(-\log \pr{ \norm{X}_{L_p[0,1]} \leq  \eps}\right) \\ \leq \inf_{D>0} ( \kappa_{\norm{.}}^\dM D^{-\gamma} + K D^{1/(\dealpha+1/2)}) =: C=C(\kappa,\kappa_{\norm{.}}^\dM).
\end{multline*}

\end{proof}

Analogously to Theorem~\ref{thm:converseone} we can prove the following in the general setup.

\begin{thm} \label{thm:conversegen} Let $0\leq\gamma<1$, $\delta\in\R$, and $\norm{.}\in\mathbf{N}(\beta,p)$.  Then
$$
-\log \pr{ \norm{X} \leq  \eps} \gap C \eps^{-\gamma} |\log \eps|^{\delta},
$$
implies
$$
-\log \pr{ \norm{X^{(M)}}_{L_2[0,1]} \leq  \eps} \gap \kappa \eps^{-1/(1/\gamma-M+\beta+1/p)} |\log \eps|^{\delta/(\gamma (1/\gamma-M+\beta+1/p))}.
$$
\end{thm}

The proof is analogous to the one of Theorem~\ref{thm:mostgen}.

\section{Results for the path regularity} \label{sec:pr}
We now come to the mentioned results on the path regularity. In many articles on small deviations for Gaussian processes the authors mention that one can read off  the path regularity from the small deviation results. However, to the knowledge of the author, no concrete result has been available so far.

The results on the path regularity follow from a modification of the proof of Theorem~\ref{thm:mostgen}. Corresponding to $\dealpha=0$ in Theorem~\ref{thm:mostgen} one obtains the following.

\begin{thm} \label{thm:regularity} Let $X^{(\dM)}\in L_2[0,1]$ a.s. with $M>0$. Then
\begin{itemize}
 \item[(i)] for any $1\leq p\leq \infty$,
$$
-\log \pr{ \norm{X}_{L_p[0,1]} \leq  \eps} \lew \eps^{-1/\dM}.
$$
\item[(ii)] Furthermore, if $\norm{.}\in\mathbf{N}(\beta,p)$ then
$$
-\log \pr{ \norm{X} \leq  \eps} \lew \eps^{-1/(\dM-\beta-1/p)}.
$$
\end{itemize}
\end{thm}

\begin{proof} We only have to show (ii). The case of $L_p$-norms, part (i), is only a special case.

If $X^{(\dM)}\in L_2[0,1]$ a.s.\ then there is a $K>0$ such that
$$
\pr{ \norm{X^{(\dM)}}_{L_2[0,1]} \leq K} =: q >0.
$$
Using the Markov inequality, we obtain
$$
q=\pr{ \norm{X^{(\dM)}}_{L_2[0,1]} \leq K}=\pr{ e^{-\frac{\lambda^2}{2}\,\norm{X^{(\dM)}}_{L_2[0,1]}^2} \geq e^{-\frac{\lambda^2 K^2}{2}}} \leq \frac{\E \exp(-\frac{\lambda^2}{2}\,\norm{X^{(\dM)}}_{L_2[0,1]}^2)}{\exp(-\frac{\lambda^2K^2}{2})}.
$$
Concerning the reproducing kernel Hilbert space with argue as in (\ref{eqn:rlappl}). Therefore, Proposition~\ref{prop:chenli} implies that, setting $H:=M-1/2$ and $1/\gamma:=H-\beta-1/p$, we have
$$
\pr{ \norm{X} \leq  \eps} \geq \pr{ \norm{R^H} \leq  \lambda \eps} \E \exp\left( -\frac{\lambda^2}{2} \norm{X^{(\dM)}}_{L_2[0,1]}^2 \right) \geq  e^{- c (\lambda \eps)^{-\gamma} }\, q e^{-\frac{\lambda^2K^2}{2}},
$$
where we used the result for Riemann-Liouville processes (\ref{eqn:rlresult}) for the first term. Setting $\lambda:=\eps^{-\gamma/(2+\gamma)}$, this gives
$$
- \log \pr{ \norm{X} \leq  \eps}\lew \eps^{-2\gamma/(2+\gamma)}=\eps^{-1/(M-\beta-1/p)}.
$$
\end{proof}

In particular, in the case of integer derivatives, we obtain the following.

\begin{cor} \label{cor:nregularity}
If the process $X$ is $n$-times differentiable with $X^{(n)}\in L_2[0,1]$ then
$$
-\log \pr{ \norm{X}_{L_p[0,1]} \leq  \eps} \lew \eps^{-1/n}.
$$
If the process is $C^\infty$ then 
\begin{equation} \label{eqn:cinfty}
\lim_{\eps\to 0} \eps^{\delta}\left( -\log \pr{ \norm{X}_{L_p[0,1]} \leq  \eps} \right) = 0,
\end{equation}
for all $\delta>0$.
\end{cor}

A close look at Theorem~\ref{thm:regularity} and Corollary~\ref{cor:nregularity} reveals the following interesting interpretation: if the process has a certain path regularity -- $X^{(\dM)}$ is in $L_2$ -- then the small deviation probability cannot be too small  -- the logarithmic small deviation probability cannot grow faster than $\eps^{-1/M}$. Conversely, if the small deviations grow too fast, then the path of the process cannot be too regular.

This is the first very concrete result of the intuitive fact that path regularity and small deviations for Gaussian processes are closely connected. We stress that beyond $\mathcal{C}^\infty$ the small deviation asymptotics has still distrinct rates giving additional information.

\begin{rem} \label{rem:notsharp} We remark that Theorem~\ref{thm:regularity} and Corollary~\ref{cor:nregularity} usually do {\it not} give the precise small deviation order even though we know the precise path regularity. Typically, $X^{(\dM)}\in L_2$ for all $\dM>\dM_0$. And thus the theorem yields
$$
-\log \pr{ \norm{X}_{L_p[0,1]} \leq  \eps} = \eps^{-1/\dM_0-o(1)}.
$$
However, in many cases one finds that the above holds without the $o(1)$ term:
$$
-\log \pr{ \norm{X}_{L_p[0,1]} \leq  \eps} \approx \eps^{-1/\dM_0}.
$$
One can think of Brownian motion itself (where $\dM_0=1/2$) or Riemann-Liouville processes. So, one can only hope to determine the small deviation asymptotics from the path regularity up to an $o(1)$ term.
\end{rem}

\section{Related questions}
\subsection{Remarks on stable processes} \label{sec:stable}
One may ask whether it is possible to extend the above results beyond the setup of Gaussian processes. This is indeed possible.

Since symmetric $\alpha$-stable processes (in the sense of \cite{ST}) can be represented as conditionally Gaussian processes, the main tool used in the proofs, Theorem~1.2 in \cite{chenli}, can be transfered. Therefore, Theorems~\ref{thm:bmint}, \ref{thm:converseone}, \ref{thm:hsb}, \ref{thm:mostgen}, and \ref{thm:conversegen} as well as Theorem~\ref{thm:regularity} and Corollary~\ref{cor:nregularity} hold true also for symmetric $\alpha$-stable processes.

However, it is easy to construct stable processes such that Theorem~\ref{thm:regularity} and Corollary~\ref{cor:nregularity} do not give sharp results. This can be seen from the following example.

\begin{exa} Let $X$ be a subfractional Brownian motion, i.e.\ $$X(t)=A^{1/2} B^H(t), \qquad t\geq 0,$$ where $A$ is a strictly positive $\alpha/2$-stable random variable and $B^H$ is a fractional Brownian motion independent of $A$. Then obviously $X$ has exactly the same a.s.\ path properties as $B^H$. So, somehow one might want to expect that $X$ should have the same small deviation order as $B^H$. This is not the case, as shown by Samorodnitsky \cite{samorodnitsky98}:
$$
-\log \pr{ \norm{X}_{L_\infty[0,1]}\leq \eps} \approx \eps^{-\frac{1}{1/\alpha-1/2+ H}} \qquad\text{vs.}\quad-\log \pr{ \norm{B^H}_{L_\infty[0,1]}\leq \eps} \sim C \eps^{-1/H}. 
$$
\end{exa}

This underlines that, partially, the small deviation rate is due to the fluctuations of the process, i.e.\ the path regularity, partially it is due to the (in this case: heavy) tail behaviour of the process, i.e.\ the amplitudes of the fluctuations.

\subsection{Eigenvalues of the covariance operator} \label{sec:eigenvalues}
We recall that small deviation probabilities in $L_2$-norm are closely connected (see e.g.\ \cite{najmc03,nazarovnikitin,nani05} and references therein) to the eigenvalues of the integral equation:
$$\lambda_n f_n(t) = \int_0^1 R(t,s) f_n(s) d s,\qquad n\geq 1,$$ where $R(t,s)=\E X(t) X(s)$ is the covariance kernel of the Gaussian process $X$. Our results imply the following.

\begin{cor} Let $\dealpha>0$ and $\debeta\in\R$. Let $(\lambda_n)$, $(\lambda_n^1)$ be the sequence of eigenvalues of the operators given by the kernels
$$
R(t,s) \qquad \text{and, respectively,}\qquad \frac{\partial^2 R}{\partial t\partial s}(t,s)
$$ Then $$\lambda_n^1\lew n^{-\dealpha} (\log n)^{\debeta}$$ implies $$\lambda_n\lew n^{-\dealpha-2} (\log n)^{\debeta}.$$ %where $c'$ can be calculated explicitly from $c$.
\end{cor}
A similar corollary can be obtained for fractional derivatives.

\subsection{Quantization} \label{sec:quantization}
We recall from \cite{dfms,grafetal} that small deviations for Gaussian processes are closely related to the quantization problem. Let $D(r|X,\norm{.})$ be the quantization error of the process $X$ w.r.t.\ the distortion given by the norm $\norm{.}$, i.e.\ for a normed space $(E,\norm{.})$, $s>0$, and $r>0$,
$$
D(r|X,\norm{.}) := \inf \left\lbrace \left( \E \min_{a\in\mathcal{C}}\norm{X-a}^s \right)^{1/s} ~:~ \mathcal{C} \subseteq E, \log \# \mathcal{C} \leq r \right\rbrace.
$$
The idea behind this quantity is that a random signal $X=X(\omega)$ has to be encoded; as a code one can use a minimizer $a(\omega)\in \mathcal{C}$ (minimizing $\min_{a\in \mathcal{C}} \|X(\omega)-a\|$); and if $\mathcal{C}$ was chosen close to optimal, this procedure gives a lowest possible mean error of the coding. 

From the above results on small deviations and the connection established in \cite{dfms,grafetal} one can obtain the following corollaries. The first gives the flavour of the more general result.

\begin{cor} Let $X$ be a centered Gaussian process on $[0,1]$ that is differentiable and such that $X'\in L_2[0,1]$ a.s.  Let $\dealpha>0$ and $\debeta\in\R$.  Then
$$
D(r|X',\norm{.}_{L_2[0,1]})\lew r^{-\dealpha} (\log r)^{\debeta}
$$
implies 
$$
D(r|X,\norm{.}_{L_p[0,1]})\lew r^{-(\dealpha+1)} (\log r)^{\debeta}
$$
for any $1\leq p\leq \infty$.
\end{cor}

This result can be used in the following way: In order to find the quantization rate of $X$ w.r.t.\ the $L_p$-norm, it suffices to find a good estimate for the (easier) quantization problem for $X'$ w.r.t.\ the $L_2$-norm. Unfortunately, this connection is \textit{not constructive}, so it is not clear how to obtain a good quantizer for $X$ in $L_p$ given that one has a good quantizer for the derivative $X'$ w.r.t.\ $L_2$ distortion. It would be interesting to find a constructive proof of this fact.

Now we come to a more general corollary for the quantization error.

\begin{cor} \label{cor:quantallg} Let $X$ be a Gaussian process on $[0,1]$, $\norm{.}\in \mathbf{N}(\beta,p)$, $M>0$, $\dealpha>0$, and $\debeta\in\R$. Then
$$
D(r|X^{(\dM)},\norm{.}_2)\lew r^{-\dealpha} (\log r)^{\debeta}
$$
implies 
$$
D(r|X,\norm{.})\lew r^{-(\dealpha+M-\beta-1/p)} (\log r)^{\debeta}.
$$
\end{cor}

In particular, our results have the following corollary corresponding to the case $\dealpha=0$ above.

\begin{cor} \label{cor:quantremainder} Let $X$ be a Gaussian process on $[0,1]$, $\norm{.}\in \mathbf{N}(\beta,p)$, and $M>0$. If $X^{(\dM)}\in L_2[0,1]$ a.s. then 
$$
D(r|X,\norm{.})\lew r^{-(M-\debeta-1/p)}.
$$
\end{cor}

Results that are very similar to Corollary~\ref{cor:quantremainder} for {\it not necessarily Gaussian processes} were obtained in Lemma~2.1 in \cite{steffenpr} (also see \cite{dereichvormoor}). The technique used there is based on entropy numbers of embeddings and is thus more robust than the Gaussian techniques employed here. However, it seems only possible to use these techniques in the sense of `remainder terms' as in Corollary~\ref{cor:quantremainder}, not the way we use it in Theorem~\ref{thm:mostgen}.

Furthermore, we refer to \cite{lp} for similar results translating the mean path regularity (given by the behaviour of the covariance function in the Gaussian case) into estimates for the quantization error. There, wavelet representations for the process are used to obtain good quantizers. It seems possible that this approach can help to extend the results of this paper from the Gaussian setup to other processes.

Remark~\ref{rem:notsharp} applies accordingly to Corollary~\ref{cor:quantremainder}. However, we stress that Corollary~\ref{cor:quantallg} is sharp.

%In order to formulate a result corresponding to the case $\dealpha=\infty$ in Corollary~\ref{cor:quantallg}, we have to use the following result (\cite{dfms}).

%\begin{lem} Let $X$ be a Gaussian random variable in $E$. Let $\norm{.}$ be a norm on $E$ and $\gamma>1$.
%Then $|\log D(r|X,\norm{.})|\gew r^\gamma$ implies $-\log \pr{\norm{X}\leq \eps} \lew |\log \eps|^{1/\gamma}$
%\end{lem}

%This yields the following corollary.

%\begin{cor}
%Let $X$ be a Gaussian process on $[0,1]$. Let $\norm{.}$ be a translation invariant, $\debeta$-self-similar, and $p$-pseudo norm on $[0,1]$. Let $\gamma>1$. Then
%$$
%|\log D(r|X^{(\dM)},\norm{.}_2)|\gew r^{-\gamma}
%$$
%implies 
%$$
%|\log D(r|X,\norm{.})|\gew r^{-\gamma}.
%$$
%\end{cor}

\subsection{Entropy numbers of operators} \label{sec:entropy}
Via the connections between small deviations and entropy numbers for linear operators (\cite{kuelbsli}, \cite{lilindeannals1999}, \cite{atpa}) we can obtain the following corollaries for operators related to Gaussian processes.

Let $M>0$ and let $K : [0,1]\times [0,1] \to \R$ be a measurable function with
$$
K(t,.)\in L_2[0,1] \qquad\text{and}\qquad \frac{\partial^{\dM} K}{\partial t^{\dM}}(t,.)\in L_2[0,1],
$$
where the fractional derivative is as defined in (\ref{eqn:fracder1}).

We consider the following operators
\begin{eqnarray*}
u : L_2[0,1] \to L_p[0,1],&\qquad & (u f)(t) := \int_0^1 K(t,s) f(s) \, d s\\
u_{\dM} : L_2[0,1] \to L_2[0,1],&\qquad & (u_{\dM} f)(t) := \int_0^1 \frac{\partial^{M} K}{\partial t^{\dM}} (t,s) f(s) \, d s.
\end{eqnarray*}

For a linear operator $u : E \to F$ between Banach spaces $E$ and $F$ one defines the entropy numbers as follows \cite{Carl}:
$$
e_n(u : E \to F) = \inf \{ \eps>0 \,:\, \exists f_1, \ldots, f_{2^{n-1}} \in F \,\forall x\in E, \norm{x}\leq 1 \,:\ \norm{u(x)-f} \leq \eps \}.
$$

The following theorem relates the (easy) $L_2[0,1]\to L_2[0,1]$ entropy numbers of $u_{\dM}$ to the (more difficult) $L_2[0,1]\to L_p[0,1]$ entropy numbers of $u$. 

\begin{thm} Let $\dealpha>0$, $\debeta\in\R$, and $1\leq p\leq \infty$. Set $E=L_p[0,1]$. If
$$
e_n(u_{\dM} : L_2[0,1] \to L_2[0,1]) \lew n^{-1/2-\dealpha} (\log n)^{\debeta}
$$
then
$$
e_n(u : L_2[0,1] \to L_p[0,1]) \lew n^{-1/2-\dealpha-M} (\log n)^{\debeta}.
$$
\end{thm}

This theorem has the following interpretation: the regularity of the kernel (namely, the fact that $\partial^\dM K/ \partial t^\dM$ is in $L_2[0,1]$) can be used to translate an (easy) $L_2\to L_2$ estimate for $u_{\dM}$ into an $L_2\to L_p$ estimate for $u$.

We remark that this theorem is \emph{not} an embedding type theorem, since it holds for all $p\geq 2$ and $1\leq p\leq 2$.

\begin{rem} Let us mention that one can treat the cases $\dealpha=0$ and $\dealpha=\infty$ accordingly.

First we look at the case $\dealpha=0$. If for some $\debeta>1/2$
$$
e_n(u_{\dM} : L_2[0,1] \to L_2[0,1]) \lew n^{-1/2} (\log n)^{-\debeta}
$$
then
$$
e_n(u : L_2[0,1] \to L_p[0,1]) \lew n^{-1/2-\dM}.
$$

%In particular, if $\partial K/\partial t$ exists and the corresponding operator $u_{\dM}$ satisfies (\ref{eqn:dudley}) then
%$$
%e_n(u) \lew n^{-3/2}.
%$$

The case $\dealpha=\infty$ can  also  be treated as above. Namely, for $\gamma>0$
$$
|\log e_n(u_{\dM} : L_2[0,1] \to L_2[0,1])| \gew n^{-\gamma} \quad \Rightarrow\quad  |\log e_n(u : L_2[0,1] \to L_p[0,1])| \gew n^{-\gamma}.
$$
\end{rem}

Similar estimates can be obtained for other norms. In particular, if the $L_p[0,1]$-norm is replaced by a norm $\norm{.}\in\mathbf{N}(\beta,p)$ then $M$ in the assertions has to be replaced by $M-\beta-1/p$.
%If
%$$
%u : L_2[0,1] \to E,\qquad (u f)(t) := \int_0^1 K(t,s) f(s) \, d s,
%$$
%where $E$ is a Banach space of functions mapping from $[0,1]$ into the real numbers with norm $\norm{.}$. We assume that $\norm{.}$ is a translation invariant, $\debeta$-self-similar, and $p$-pseudo additive norm.

%\begin{thm} Under the above conditions, let $\dealpha>0$ and $\debeta\in\R$. If
%$$
%e_n(u_{\dM} : L_2[0,1] \to L_2[0,1]) \lew n^{-1/2-\dealpha} (\log n)^{\debeta}
%$$
%then
%$$
%e_n(u : L_2[0,1] \to E) \lew n^{-1/2-\dealpha-M+\debeta+1/p} (\log n)^{\debeta }.
%$$
%\end{thm}
%
%Again this theorem can be used to translate an (easy) $L_2\to L_2$ estimate for $u_{\dM}$ into an $L_2\to E$ estimate for $u$, where for $E$ one can take basically any classical Banach space.

\subsection{Examples} \label{sec:examples}

The first example is the integrated fractional Brownian motion.

\begin{exa} Consider the case when $X$ is an integrated fractional Brownian motion. Then one obtains
$$
-\log \pr{ \norm{X}_{L_p[0,1]} \leq  \eps} \lew \eps^{-\frac{1}{H + 1}}.
$$
See e.g.\ \cite{nazarovnikitin}. The lower bound can be obtained by comparing $\norm{X}_{L_p[0,1]}$ to the $L_2$-norm of $X$ itself, which has the same order. Note that the $H$ comes from fractional Brownian motion (and there from being H\"{o}lder up to $H$) and the $1$ from one integration.

Consider the case where $X$ is $m$-times integrated fractional Brownian motion. Then $X^{(m)}$ is precisely fractional Brownian motion and %and $\kappa_p^H$ ($H=m-1/2$) is the small deviation constant for $(m-1)$-times integrated Brownian motion w.r.t.\ $L_p$-norm.
Theorem~\ref{thm:hsb} states that a lower bound for the $L_2$ small deviations of fractional Brownian motion can be translated into a lower bound for the $L_p$ small deviations of $m$-times integrated fractional Brownian motion. The result is
$$
-\log \pr{ \norm{X}_{L_p[0,1]} \leq  \eps} \lap \eps^{-1/(H+m)},
$$
which is the correct order, cf.\ \cite{nazarovnikitin} and Theorem~1.3 in \cite{chenli} for the Brownian case. The lower bound follows by comparison to the $L_2$ small deviations of $X$ itself. This method was already used in \cite{chenli} for the Brownian case $H=1/2$.
\end{exa}

The second example concerns a conjecture by Lifshits and Simon \cite{lifshitssimon2004}.

\begin{exa} Set
\begin{equation} \label{eqn:RL}
R^H(t):=\int_0^t (t-s)^{H-1/2} \, d B(s),\quad B^H(t):=R^H + \int_{-\infty}^0 (t-s)^{H-1/2} - (-s)^{H-1/2}\, d B(s), 
\end{equation}
where $B$ is a Brownian motion. Then $B^H$ is a fractional Brownian motion and $R^H$ is a Riemann-Liouville process. Then one can look at the difference process $M=R^H-B^H$. This process is $\mathcal{C}^\infty$; and thus we have by Corollary~\ref{cor:nregularity}
$$
\lim_{\eps\to 0} \eps^{\delta}\left( -\log \pr{ \norm{X}_{L_p[0,1]} \leq  \eps} \right) = 0,\qquad \text{for any $\delta>0$,}
$$
as conjectured by Lifshits and Simon  \cite{lifshitssimon2004}. In particular, this makes the proof of Theorem~12 in \cite{lifshitssimon2004} redundant. The same result is true if $B^H$ is a so-called linear fractional stable motion and $R^H$ is a symmetric stable Riemann-Liouville process, see \cite{lifshitssimon2004}. \end{exa}

The next example is of similar type.

\begin{exa}
We consider so-called stable convolutions \cite{kwapienrosinski}, i.e.\ processes
$$
X(t)=\int_0^t f(t-s) \, d B(s),\qquad t\geq 0,
$$
where $B$ is Brownian motion and $f$ is a smooth function except possibly at zero, where we assume $f(x)=x^{H-1/2} g(x)$ with $H>0$ and a function $g$ with $g(0)=1$. This setup was studied in \cite{asimon}. The goal is to show that $X$ has the same small deviation order (and path regularity) as a Riemann-Liouville process $R^H$.

We assume that $g$ has the Taylor expansion $g(x)=\sum_{n=0}^\infty a_n x^n$, $0\leq x \leq 1$. Let $R^H$ be a Riemann-Liouville process as defined in (\ref{eqn:RL}). Then we can represent $X$ as
\begin{multline*}
X(t)=\int_0^t (t-s)^{H-1/2} g(t-s) \, d B(s) \\ = \int_0^t (t-s)^{H-1/2} \, d B(s) + \sum_{n=1}^\infty \int_0^t (t-s)^{H-1/2+n} a_n \, d B(s) = R^H(t) + \sum_{n=1}^\infty a_n R^{H+n}(t).
\end{multline*}
Here the processes $R^{H+n}$, $n\geq 0$, are \emph{not} independent but rather obtained by integrating the same Brownian motion as in (\ref{eqn:RL}).

Now the weak decorrelation inequality of Li \cite{lidecorr} implies the following. In order to show that $X$ and $R^H$ have the same small deviation rate it suffices to know that the difference process $X-R^H$ has a lower small deviation order than $R^H$. However, this follows easily from Theorem~\ref{thm:regularity}, since we know that $X-R^H$ is smoother than $R^H$.

This gives a significantly shorter proof of Theorem~2.1 in \cite{asimon}.  The non-Gaussian stable case also treated in \cite{asimon} cannot be handled this way due to the absence of the decorrelation inequality.
\end{exa}

The above technique can be applied in general: assume $X$ can be represented as $X=Y+Z$ with $Y$ and $Z$ being not necessarily independent Gaussian processes. If $Z$ is smoother than $Y$ we can use Theorem~\ref{thm:regularity} to show that $Z$ has a lower small deviation rate and then Li's weak decorrelation inequality \cite{lidecorr} to show that $X$ and $Y$ have the same small deviations. We summarize this in the following corollary.

\begin{cor} \label{cor:remainderterm}
Let $X$ be a centered Gaussian random variable with values in a normed space $(E,\norm{.})$, where $\norm{.}\in\mathbf{N}(\beta,p)$. Let $\kappa>0$, $\gamma>0$, and $\ell$ be a slowly varying function. Assume we can represent $X$ as $X=Y+Z$ with $Y$ and $Z$ also Gaussian random variables in $E$ (not necessarily independent). If
$$
-\log \pr{\norm{Y}\leq \eps} \sim \kappa \eps^{-\gamma} \ell(\eps)
$$
and $Z^{(M)}\in L_2[0,1]$ for some $M>1/\gamma-\beta-1/p$, then
$$
-\log \pr{\norm{X}\leq \eps} \sim \kappa \eps^{-\gamma} \ell(\eps).
$$ 
The same holds for weak asymptotics.
\end{cor}

We finish with another example of stable processes.

\begin{exa} We consider integrated linear fractional stable motions (LFSMs), cf.\ \cite{ST}. Let $a,b\in\R$ and $H>0$ and define
$$
X^{a,b}(t):=\int_\R a((t-s)_+^{H-1/\alpha} - (-s)_+^{H-1/\alpha}) + b((t-s)_-^{H-1/\dealpha} - (-s)_-^{H-1/\dealpha}) \, d Z(s),
$$
where $Z$ is a symmetric $\alpha$-stable L\'{e}vy process, $x_+:=\max(x,0)$ and $x_-:=(-x)_+$. The small ball problem for these processes was treated in \cite{lifshitssimon2004} ($a=1$, $b=0$) and in \cite{aspl1} (general $a,b$) to the end that $X^{a,b}$ has the same small deviation rate (for any $\norm{.}\in\mathbf{N}(\beta,p)$) as the corresponding Riemann-Liouville process. From Theorem~\ref{thm:mostgen} follows a bound for the small deviation rate of integrated $X^{a,b}$.
\end{exa}

%\hypertarget{vd}{~}\pdfbookmark[1]{References}{vd}
\addcontentsline{toc}{section}{References}
  \bibliographystyle{alpha}

\begin{thebibliography}{DFMS03}

\bibitem[AILvZ08]{atpa}
Frank Aurzada, Ildar~A. Ibragimov, Mikhail Lifshits, and Harry van Zanten.
\newblock Small deviations of smooth stationary {G}aussian processes.
\newblock To appear in: {\it Theory of Probability and its Applications},
  arXiv:0803.4238v1, 2008.

\bibitem[AS07]{asimon}
Frank Aurzada and Thomas Simon.
\newblock Small ball probabilities for stable convolutions.
\newblock {\em ESAIM Probab. Stat.}, 11:327--343 (electronic), 2007.

\bibitem[Aur08]{aspl1}
Frank Aurzada.
\newblock Small deviations for stable processes via compactness properties of
  the parameter set.
\newblock {\em Statist. Probab. Lett.}, 78(6):577--581, 2008.

\bibitem[BGT89]{bgt}
Nick~H. Bingham, Charles~M. Goldie, and J\'{o}zef~L. Teugels.
\newblock {\em Regular variation}, volume~27 of {\em Encyclopedia of
  Mathematics and its Applications}.
\newblock Cambridge University Press, Cambridge, 1989.

\bibitem[CL03]{chenli}
Xia Chen and Wenbo~V. Li.
\newblock Quadratic functionals and small ball probabilities for the {$m$}-fold
  integrated {B}rownian motion.
\newblock {\em Ann. Probab.}, 31(2):1052--1077, 2003.

\bibitem[CS90]{Carl}
Bernd Carl and Irmtraud Stephani.
\newblock {\em Entropy, compactness and the approximation of operators},
  volume~98 of {\em Cambridge Tracts in Mathematics}.
\newblock Cambridge University Press, Cambridge, 1990.

\bibitem[Der08]{steffenpr}
Steffen Dereich.
\newblock The coding complexity of diffusion processes under supremum norm
  distortion.
\newblock {\em Stochastic Process. Appl.}, 118(6):917--937, 2008.

\bibitem[DFMS03]{dfms}
Steffen Dereich, Franz Fehringer, Anis Matoussi, and Michael Scheutzow.
\newblock On the link between small ball probabilities and the quantization
  problem for {G}aussian measures on {B}anach spaces.
\newblock {\em J. Theoret. Probab.}, 16(1):249--265, 2003.

\bibitem[DV04]{dereichvormoor}
Steffen Dereich and Christian Vormoor.
\newblock The high resolution vector quantization problem with {O}rlicz norm
  distortion.
\newblock Preprint, 2004.

\bibitem[FT04]{filletal}
James~Allen Fill and Fred Torcaso.
\newblock Asymptotic analysis via {M}ellin transforms for small deviations in
  {$L\sp 2$}-norm of integrated {B}rownian sheets.
\newblock {\em Probab. Theory Related Fields}, 130(2):259--288, 2004.

\bibitem[Gao08]{gaomr}
Fuchang Gao.
\newblock Entropy estimate for {$k$}-monotone functions via small ball
  probability of integrated {B}rownian motion.
\newblock {\em Electron. Commun. Probab.}, 13:121--130, 2008.

\bibitem[GHLT03]{gaoetal}
Fuchang Gao, Jan Hannig, Tzong-Yow Lee, and Fred Torcaso.
\newblock Laplace transforms via {H}adamard factorization.
\newblock {\em Electron. J. Probab.}, 8:no. 13, 20 pp. (electronic), 2003.

\bibitem[GHT03]{gaohannigtorcasso}
Fuchang Gao, Jan Hannig, and Fred Torcaso.
\newblock Integrated {B}rownian motions and exact {$L\sb 2$}-small balls.
\newblock {\em Ann. Probab.}, 31(3):1320--1337, 2003.

\bibitem[GLP03]{grafetal}
Siegfried Graf, Harald Luschgy, and Gilles Pag{\`e}s.
\newblock Functional quantization and small ball probabilities for {G}aussian
  processes.
\newblock {\em J. Theoret. Probab.}, 16(4):1047--1062 (2004), 2003.

\bibitem[KL93]{kuelbsli}
James Kuelbs and Wenbo~V. Li.
\newblock Metric entropy and the small ball problem for {G}aussian measures.
\newblock {\em J. Funct. Anal.}, 116(1):133--157, 1993.

\bibitem[KMR06]{kwapienrosinski}
Stanis{\l}aw Kwapie{\'n}, Michael~B. Marcus, and Jan Rosi{\'n}ski.
\newblock Two results on continuity and boundedness of stochastic convolutions.
\newblock {\em Ann. Inst. H. Poincar\'e Probab. Statist.}, 42(5):553--566,
  2006.

\bibitem[KS98]{khoshshi}
Davar Khoshnevisan and Zhan Shi.
\newblock Chung's law for integrated {B}rownian motion.
\newblock {\em Trans. Amer. Math. Soc.}, 350(10):4253--4264, 1998.

\bibitem[Li99]{lidecorr}
Wenbo~V. Li.
\newblock A {G}aussian correlation inequality and its applications to small
  ball probabilities.
\newblock {\em Electron. Comm. Probab.}, 4:111--118 (electronic), 1999.

\bibitem[Lif95]{lifshits}
Mikhail Lifshits.
\newblock {\em {G}aussian {R}andom {F}unctions}.
\newblock Mathematics and {I}ts {A}pplications. Kluwer Academic Publishers,
  Dordrecht, 1995.

\bibitem[Lif08]{sdreferences}
Mikhail Lifshits.
\newblock Bibliography on small deviation probabilities.
\newblock available from:\\ {\tt
  http://www.proba.jussieu.fr/pageperso/smalldev/biblio.html}, Dec. 2008.

\bibitem[LL99]{lilindeannals1999}
Wenbo~V. Li and Werner Linde.
\newblock Approximation, metric entropy and small ball estimates for {G}aussian
  measures.
\newblock {\em Ann. Probab.}, 27(3):1556--1578, 1999.

\bibitem[LP08]{lp}
Harald Luschgy and Gilles Pag{\`e}s.
\newblock Functional quantization rate and mean regularity of processes with an
  application to {L}\'evy processes.
\newblock {\em Ann. Appl. Probab.}, 18(2):427--469, 2008.

\bibitem[LS01]{lishao-overview}
Wenbo~V. Li and Qi-Man Shao.
\newblock Gaussian processes: inequalities, small ball probabilities and
  applications.
\newblock In {\em Stochastic processes: theory and methods}, volume~19 of {\em
  Handbook of Statist.}, pages 533--597, Amsterdam, 2001. North-Holland.

\bibitem[LS05]{lifshitssimon2004}
Mikhail~A. Lifshits and Thomas Simon.
\newblock Small deviations for fractional stable processes.
\newblock {\em Ann. Inst. H. Poincar\'e Probab. Statist.}, 41(4):725--752,
  2005.

\bibitem[Naz03]{najmc03}
Alexander~I. Nazarov.
\newblock On the sharp constant in the small ball asymptotics of some
  {G}aussian processes under {$L\sb 2$}-norm.
\newblock {\em J. Math. Sci. (N. Y.)}, 117(3):4185--4210, 2003.
\newblock Nonlinear problems and function theory.

\bibitem[NN04a]{nazarovnikitin}
Alexander~I. Nazarov and Yakov~Yu. Nikitin.
\newblock Exact {$L\sb 2$}-small ball behavior of integrated {G}aussian
  processes and spectral asymptotics of boundary value problems.
\newblock {\em Probab. Theory Related Fields}, 129(4):469--494, 2004.

\bibitem[NN04b]{nani05}
Alexander~I. Nazarov and Yakov~Yu. Nikitin.
\newblock Logarithmic asymptotics of small deviations in the {$L\sb 2$}-norm
  for some fractional {G}aussian processes.
\newblock {\em Teor. Veroyatn. Primen.}, 49(4):695--711, 2004.

\bibitem[Sam98]{samorodnitsky98}
Gennady Samorodnitsky.
\newblock Lower tails of self-similar stable processes.
\newblock {\em Bernoulli}, 4(1):127--142, 1998.

\bibitem[SKM93]{skm}
Stefan~G. Samko, Anatoly~A. Kilbas, and Oleg~I. Marichev.
\newblock {\em Fractional integrals and derivatives}.
\newblock Gordon and Breach Science Publishers, Yverdon, 1993.

\bibitem[ST94]{ST}
Gennady Samorodnitsky and Murad~S. Taqqu.
\newblock {\em Stable non-{G}aussian random processes}.
\newblock Stochastic Modeling. Chapman \& Hall, New York, 1994.

\end{thebibliography}

\end{document}